\documentclass[11pt,reqno]{amsart}

\usepackage{amsfonts}
\usepackage{indentfirst}
\usepackage{graphicx}
\usepackage{amssymb}
\usepackage{color,xcolor}
\usepackage{graphicx,epstopdf}
\usepackage{epsfig}
\usepackage{subfigure}
\usepackage{caption}
\usepackage{mathrsfs}
\usepackage{bbm}
\usepackage{chngpage}
\usepackage{graphicx}
\usepackage{cite}
\usepackage{multirow}
\usepackage{multicol}
\usepackage{dsfont}
\usepackage{bm}
\usepackage{amssymb,amsmath,amsthm,mathrsfs}%
\usepackage{exscale}
\usepackage{tikz-cd}
\usepackage{relsize}
\usepackage{amssymb}
\usepackage{hyperref}
\usepackage{url}
\usepackage{tikz-cd}
\usepackage{comment}
\usepackage[misc,geometry]{ifsym} 
\allowdisplaybreaks

\hypersetup{hypertex=blue,
            colorlinks=true,
            linkcolor=blue,
            anchorcolor=blue,
            citecolor=red}

\theoremstyle{definition}
\newtheorem{theorem}{Theorem}[section]

\newtheorem{assumption}[theorem]{Assumption}

\newtheorem{remark}{Remark}[section]
\newtheorem{lemma}{Lemma}[section]

\numberwithin{equation}{section}%
\numberwithin{table}{section}%
\numberwithin{figure}{section}

\def\3bar{{|\hspace{-.02in}|\hspace{-.02in}|}}

\newcommand\grad{\operatorname{grad}}
\renewcommand\div{\operatorname{div}}

\newcommand\curl{\operatorname{curl}}

\newcommand\ran{\operatorname{ran}}

\newcommand\I{{\operatorname I}}

\def\d{\text{d}}

\begin{document}
\title[Quadratic and cubic Lagrange finite elements]{Quadratic and cubic Lagrange finite elements for mixed Laplace eigenvalue problems on criss-cross meshes}

\author{Kaibo Hu}
\address{School of Mathematics, the University of Edinburgh, James Clerk Maxwell Building, Peter Guthrie Tait Rd, Edinburgh EH9 3FD, UK.}
\email{kaibo.hu@ed.ac.uk}

\author{Jiguang Sun}
\address{Department of Mathematical Sciences, Michigan Technological University, Houghton, MI 49931, USA. }
\email{jiguangs@mtu.edu}

\author{Qian Zhang}
\address{Department of Mathematical Sciences, Michigan Technological University, Houghton, MI 49931, USA. }
\email{qzhang15@mtu.edu}


\begin{abstract} 
In \cite{boffi2000problem}, it was shown that the linear Lagrange element space on criss-cross meshes and its divergence exhibit spurious eigenvalues when applied in the mixed formulation of the Laplace eigenvalue problem, despite satisfying both the inf-sup condition and ellipticity on the discrete kernel. The lack of a Fortin interpolation is responsible for the spurious eigenvalues produced by the linear Lagrange space. In contrast,  results in \cite{drab104} confirm that quartic and higher-order Lagrange elements do not yield spurious eigenvalues on general meshes without nearly singular vertices, including criss-cross meshes as a special case.  In this paper, we investigate quadratic and cubic Lagrange elements on criss-cross meshes. We prove the convergence of discrete eigenvalues by fitting the Lagrange elements on criss-cross meshes into a complex and constructing a Fortin interpolation. As a by-product, we construct bounded commuting projections for the finite element Stokes complex, which induces isomorphisms between cohomologies of the continuous and discrete complexes. We provide numerical examples to validate the theoretical results. 
\end{abstract}	
\maketitle
\section{Introduction}
Let $\Omega\subset \mathbb R^2$ be a simply-connected Lipschitz polygonal domain. The Laplace eigenvalue problem is to seek $u\in H_0^1(\Omega)$ such that
\begin{align}\label{eigP}
	-\Delta u &=\lambda u\text{ in }\Omega.
\end{align}
Introducing $\bm\sigma = \grad u$,  we can obtain the mixed formulation:
find $(\bm \sigma, u, \lambda)$ in $H(\div;\Omega)\times L^2(\Omega)\times \mathbb R$ such that
\begin{eqnarray}
\begin{split}\label{vf}
	(\bm \sigma,\bm \tau)+(u,\div\bm \tau) = &0 &  \forall \bm\tau\in H(\div;\Omega),\\
	(\div\bm \sigma, v) =& -\lambda(u,v)& \forall v\in L^2(\Omega).
\end{split}
\end{eqnarray}
Given finite dimensional subspaces $V_h$ and $W_h$ of $H(\div;\Omega)$ and $L^2(\Omega)$, respectively, the mixed finite element approximation reads: find 
$(\bm \sigma_h, u_h, \lambda_h)$ in $V_h\times W_h\times \mathbb R$ such that
\begin{eqnarray}
\begin{split}\label{fem0}
	(\bm \sigma_h,\bm \tau_h)+(u_h,\div\bm \tau_h) &= 0 \ & \forall \bm\tau_h\in V_h,\\
	(\div\bm \sigma_h, v_h) &= -\lambda_h(u_h,v_h)\ &\forall v_h\in W_h.	
\end{split}
\end{eqnarray}
It is shown in \cite{boffi2000problem} that inappropriate choices of finite element pair $V_h$ and $W_h$ would lead to spurious eigenvalues.
To ensure the convergence of the finite element approximation \eqref{fem0} to \eqref{vf}, it is crucial that the spaces $V_h$ and $W_h$ satisfy the inf-sup condition, ellipticity on the discrete kernel, as well as an additional condition related to the Fortin interpolation. The N\'ed\'elec finite element space of order $k$ for $V_h$ and the piecewise polynomial space of order $k-1$ for $W_h$ satisfy these conditions, and can therefore produce correctly convergent eigenvalues \cite{boffi2010finite,boffi1999computational}.  

In contrast to the N\'ed\'elec finite elements, the Lagrange finite element space of order $k$ for $V_h$ and its divergence for $W_h$ produce correctly convergent eigenvalues only  on spacial triangulations. 
In \cite{drab104},  Boffi et al. rigorously proved the convergence of such a pair $V_h-W_h$ for the Maxwell eigenvalue problem on the Powell-Sabin triangulation when $k=1$, Clough-Tocher triangulation when $k=2$, and general shape-regular triangulations without nearly singular vertices when $k\geq 4$. 
They achieve this by constructing Fortin-type interpolations.
As Problem \eqref{vf} is equivalent to the mixed form of the Maxwell eigenvalue problem considered in \cite{drab104} by rotating $\bm\sigma$ by $-\pi/2$, the results in \cite{drab104} indicate that the quartic and higher-order Lagrange finite element spaces lead to correctly convergent finite element schemes for Problem \eqref{vf} on criss-cross meshes. Whereas, as shown in \cite{boffi2000problem}, the linear Lagrange element space on the criss-cross mesh exhibits spurious eigenvalues due to the lack of a Fortin interpolation. To the best of our knowledge, there have been no previous studies investigating the behavior of quadratic and cubic Lagrange elements on criss-cross meshes for problem \eqref{vf}. This paper aims to fill this gap.
 

In this paper, we establish the convergence of discrete eigenvalues computed from \eqref{fem0} with quadratic and cubic Lagrange elements on criss-cross meshes for $V_h$.
Our analysis relies on the existence of bounded Fortin operators. To construct the Fortin interpolation, we fit the quadratic and cubic Lagrange element spaces $V_h$ on the criss-cross mesh into a subcomplex of the Stokes complex, as shown in the following diagram,
\begin{equation}\label{H2complex2-2D}
\begin{tikzcd}
0 \arrow[r] 
& \mathbb{R} \arrow[r,"\subset"]  & H^2(\Omega)\arrow[d ] \arrow[r,"\curl"]  & \bm H^1(\Omega)\arrow[d ]\arrow[r,"\div"]  & L^2(\Omega)\arrow[d ]\arrow[r]& 0\\
 0 \arrow[r] 
& \mathbb{R} \arrow[r,"\subset"]  & \Sigma_h \arrow[r,"\curl"]  & V_h\arrow[r,"\div"] &\div V_h\arrow[r] & 0.
\end{tikzcd}
\end{equation}
The $H^2$-conforming finite element space $\Sigma_h$ on the criss-cross mesh is available in the literature \cite{lai2007spline} and  $V_h$-$\div V_h$ is shown to be inf-sup stable for the Stokes problem \cite{arnold1992quadratic}. Nevertheless, to the best of our knowledge, the sequence \eqref{H2complex2-2D}, which links these spaces, has not been presented yet. 
Using the finite element subcomplex \eqref{H2complex2-2D}, we can redefine the degrees of freedom (DOFs) for  $V_h$. The canonical interpolation defined by the new DOFs satisfies the necessary commutativity, but it cannot serve as a Fortin interpolation for Problem \eqref{fem0} since the boundedness requires more smoothness. We will use the Scott-Zhang interpolation to fix this boundedness issue. 
We note that our approach differs from \cite{drab104}, where a Fortin-type interpolation is constructed for functions that have piecewise-polynomial divergence. 
To use the mixed formulation \eqref{fem0}, we also provide a characterization of the space $\div V_h$. 
Moreover, the cohomology of finite element spaces plays a crucial role in numerical computation for many problems (see, e.g., \cite{arnold2018finite,arnold2006finite,alonso2018finite,rodriguez2013construction}). 
Our Fortin interpolation, together with a Scott-Zhang type interpolation, leads to bounded commuting projections for complex \eqref{H2complex2-2D}, which induces isomorphisms between the cohomologies of \eqref{H2complex2-2D}.

The pair $V_h-\div V_h$ in the diagram can also be used for solving the Stokes problem.  In \cite{arnold1992quadratic}, Arnold and Qin use the quadratic Lagrange element space $V_h$ for the Stokes problem. They used a piecewise polynomial space for pressure, which includes a spurious pressure mode. Consequently, a post-processing procedure is necessary to remove the spurious pressure mode from the numerical pressure. With the explicit characterization of $\div V_h$ in our work, we can use the pair $V_h - \div V_h$ to solve the Stokes problem, and the numerical pressure will converge without the need for post-processing.

The remaining part of the paper is organized as follows. In Section 2, we present notation, the discrete eigenvalue problem and an equivalent formulation, the source problem and solution operators, and a general theoretical framework.
In Section 3, we present the definition of each space in the finite element Stokes complex, characterize the space $\div V_h$, and prove the cohomology of the complex. 
In Section 4, we construct a Fortin interpolation that satisfies desired properties. In Section 5, we construct bounded commuting projections for \eqref{H2complex2-2D} and show their approximation property. In Section 6, we provide numerical examples to validate our convergence analysis. 

\section{Setting of the Problem}
\subsection{Notation}
Let $\mathcal Q_h$ be a partition of $\Omega$ with convex quadrilaterals. Each $Q\in\mathcal Q_h$ is further split into four triangles by the two diagonals of $Q$. We denote the triangulation of $Q$ as $\mathcal T_h^Q$. Let $\mathcal T_h$ be the partition of $\Omega$ with triangles in $\cup_{Q\in \mathcal Q_h}\mathcal T_h^Q$. We denote by $h_T$ the diameter of an element $T\in  \mathcal T_h$. Define $h= \max_{T\in\mathcal T_h}h_T$ to be the mesh size of $\mathcal T_h$. We define $\mathcal V_h$ to be the set of vertices in the mesh $\mathcal Q_h$. For $Q\in \mathcal Q_h$, we denote by $\mathcal V_h(Q)$ and $\mathcal E_h(Q)$ the sets of vertices and edges of $Q$.  Denote by $\bm n_Q$ the unit outward normal vector to $Q$. For an edge $e$ of $T\in\mathcal T_h$, we denote by $\bm n_e$ the unit normal vector to $e$.
For $Q\in \mathcal Q_h$, we define $h_Q = \max_{T\in \mathcal T_h^Q}h_T$. 

We suppose the triangles in $\mathcal T_h$ are shape-regular, i.e.,
\[\frac{h_T}{\rho_T}\leq \sigma_0,\]
where $\rho_T$ is the radius of the largest closed ball contained in $\overline T.$

For any $Q\in \mathcal Q_h$, we can map $\hat Q=(-1,1)\times(-1,1)$ to $Q$ by
\[{\bm x}=F_Q(\hat{ \bm x}),\]
where $F_Q=F_i$ on $T_i$ with $F_i(\hat{\bm x})=B_i\hat{\bm x}+\bm b_i$ mapping $\hat T_i$ to $T_i$, $i=1,2,3,4$. See Figure \ref{Q-hatQ} for $\hat T_i$ and $T_i$. 
\begin{figure}[h!]
\begin{center} 
\setlength{\unitlength}{1.5cm}

\begin{picture}(5,2)(0.5,-0)
\put(2,0){
\begin{picture}(2,0)
\put(-3, 0){\line(0,2){2}} 
\put(-3, 2){\line(2,0){2}} 
\put(-3,0){\line(2,0){2}}
\put(-1,0){\line(0,2){2}}
\put(-3,0){\line(1,1){2}}
\put(-3,2){\line(1,-1){2}}
\put(-2.13,0.3){$\hat T_1$}
\put(-2.13,1.5){$\hat T_3$}
\put(-2.7,0.95){$\hat T_2$}
\put(-1.65,0.95){$\hat T_4$}
\put(0.4,1){$\Longrightarrow$}
\end{picture}}

\put(7,0){
\begin{picture}(2,0)
\put(-3, 0){\line(1,4){0.5}} 
\put(-3, 0){\line(5,4){2.5}} 
\put(-2.5, 2){\line(1,0){2}} 
\put(-2.5, 2){\line(5,-4){2.5}}
\put(-0.5, 2){\line(1,-4){0.5}}
\put(-3,0){\line(2,0){3}}
\put(-1.65,0.5){$T_1$}
\put(-1.65,1.6){$T_3$}
\put(-2.4,1.05){$T_2$}
\put(-1,1.05){$T_4$}
\end{picture}}
\end{picture}

\end{center} 
\caption{$\hat Q$ (left) and $ Q$ (right)}\label{Q-hatQ}
\end{figure}
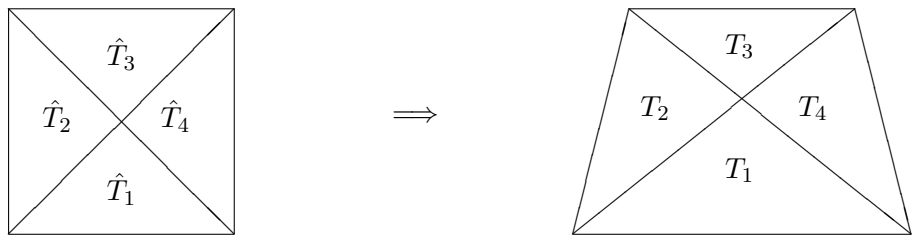

For  a 
simply-connected sub-domain $D\subset\Omega$, we adopt standard notation for Sobolev spaces such as $H^m(D)$ equipped with the norm $\left\|\cdot\right\|_{m,D}$ and the semi-norm $\left|\cdot\right|_{m,D}$. If $m=0$,  the space $H^0 (D)$ coincides with $ L^2(D)$ equipped with the norm $\|\cdot\|_{D}$, and when $D=\Omega$, we drop the subscript $D$. We use  $\bm H^m(D)$   and ${\bm L}^2(D)$ to denote the vector-valued Sobolev spaces $\left[H^m(D)\right]^2$ and $\left[L^2(D)\right]^2$.


We define
\begin{align*}
&H(\text{div};\Omega):=\{\bm u \in {\bm L}^2(\Omega):\; \div \bm u \in L^2(\Omega)\},
\end{align*}
with the scalar products and norms 
$(\bm u,\bm v)_{H(\div;\Omega)}=(\bm u,\bm v)+(\div\bm u,\div \bm v)$
and
$\left\|\bm u\right\|_{H(\div;\Omega)}=\sqrt{(\bm u,\bm u)_{H(\div;\Omega)}}.$


We use $P_k(D)$ to denote the space of polynomials on $D$ with degree of no larger than $k$ and $\bm P_k(D)=\left[P_k(D)\right]^2$. We denote $\widetilde P_k(D)$ as the space of homogeneous polynomials of order $k$.

\subsection{Discrete Eigenvalue Problem}
%

We use a supscript $k$ to denote the polynomial degree of the space $V_h$,
\begin{eqnarray*}
	V_h^{k} &=& \{ \bm u \in \bm H^1(\Omega): \bm u|_T \in \bm P_{k}(T) \text{ for any }T\in \mathcal T_h\},
\end{eqnarray*}
	and let \[W_h^{k-1}=\div V_h^{k}.\]

The finite element scheme is to find $(\bm \sigma_h, u_h, \lambda_h)$ in $V_h^k\times W_h^{k-1}\times \mathbb R$ such that
\begin{eqnarray}
\begin{split}\label{fem1}
	(\bm \sigma_h,\bm \tau_h)+(u_h,\div\bm \tau_h) &= 0 \ & \forall \bm\tau_h\in V_h^k,\\
	(\div\bm \sigma_h, v_h) &= -\lambda_h(u_h,v_h)\ &\forall v_h\in W_h^{k-1}.	
\end{split}
\end{eqnarray}
In addition to the primal formulation which is to find $(u,\lambda)\in H^1(\Omega)\times \mathbb R$ such that
\begin{eqnarray}\label{primal-v}
	(\grad u,\grad v)=\lambda(u,v) & \quad &\forall u\in H^1(\Omega), 
\end{eqnarray}
by taking $v=\div\bm\tau$ in \eqref{vf}, we can get another equivalent formulation: find $(\bm \sigma, \lambda)$ in $H(\div;\Omega)\times \mathbb R$ such that
\begin{eqnarray}\label{vf2}
	(\div\bm\sigma,\div\bm \tau)=\lambda(\bm \sigma,\bm \tau) & \quad &\forall \bm\tau\in H(\div;\Omega).
\end{eqnarray}
An equivalent scheme is to find $(\bm \sigma_h,  \lambda_h)$ in $V_h^k\times \mathbb R$ such that
\begin{eqnarray}\label{fem2}	
	(\div\bm\sigma_h,\div\bm \tau_h)=\lambda(\bm \sigma_h,\bm \tau_h) & \quad &\forall \bm\tau_h\in V_h^k.
	\end{eqnarray}
With this formulation, we can compute the numerical eigenvalues without characterizing $W_h^{k-1}.$

\subsection{Source Problem}
We will need the correponding source problem for the analysis. 
We define the solution operators $\bm A:L^2(\Omega)\rightarrow H(\div;\Omega)$ and $T:L^2(\Omega)\rightarrow L^2(\Omega)$ such that for a given $f\in L^2(\Omega)$, there holds
\begin{align}\label{source1}
\begin{split}
		(\bm Af,\bm \tau)+(Tf,\div\bm \tau) &= 0 \quad \forall \bm\tau\in H(\div;\Omega),\\
	    (\div\bm Af, v)& = -(f,v)\quad \forall v\in L^2(\Omega).
\end{split}
\end{align}

Similarly, 
we define the discrete solution operators $\bm A_h:L^2(\Omega)\rightarrow V_h^k$ and $T_h:L^2(\Omega)\rightarrow W_h^{k-1}$ such that for a given $f\in L^2(\Omega)$, there holds
\begin{align}\label{source2}
\begin{split}
		(\bm A_hf,\bm \tau)+(T_hf,\div\bm \tau) &= 0 \quad \forall \bm\tau\in V_h^k,\\
	(\div\bm A_hf, v)& = -(f,v)\quad \forall v\in W_h^{k-1}.
\end{split}
\end{align}
\begin{lemma}\label{regularity}
For $f\in L^2(\Omega)$, there exists a $\delta\in (0,1/2]$ depending on the interior angles of $\Omega$ such that $\bm Af\in \bm H^{1/2+\delta}(\Omega)$ and $Tf\in H^{3/2+\delta}(\Omega)$. Moreover,
\begin{align*}
\|\bm Af\|_{1/2+\delta}\leq C\|f\|,\\
\|Tf\|_{3/2+\delta}\leq C\|f\|.
\end{align*}
\end{lemma}
\begin{proof}
	From \eqref{source1}, we can obtain $\bm A f =\grad T f$ and $-\div \bm A f = f,$ which implies $-\Delta Tf =f$. Moreover, from the first equation of \eqref{source1}, we can get $Tf=0$ on $\partial\Omega$. By the regularity estimate of the Laplace problem \cite{Monk2003}, we have $Tf\in H^{3/2+\delta}(\Omega)$ and 
	$\|Tf\|_{3/2+\delta}\leq C\|f\|$. Then $\bm A f =\grad T f \in \bm H^{1/2+\delta}(\Omega)$ and $\|\bm Af\|_{1/2+\delta}\leq \|Tf\|_{3/2+\delta}\leq C\|f\|$.
\end{proof}

From Lemma \ref{regularity}, we can see that 
\begin{eqnarray*}
	\bm A(L^2(\Omega)) = \bm H^{1/2+\delta}(\Omega)\cap H(\div;\Omega), &\ & T(L^2(\Omega)) = H^{3/2+\delta}(\Omega).
\end{eqnarray*}

\subsection{Theoretical Framework}

Define
\[V=\bm H^{1/2+\delta}(\Omega)\cap H(\div;\Omega),\]
and equip it with norm
\[\|\cdot\|_V=\Big(\|\cdot\|_{1/2+\delta}^2+\|\cdot\|_{H(\div;\Omega)}^2\Big)^{1/2}.\]
\begin{assumption}\label{assump}We assume that
\begin{itemize}
    \item[\textbf{H1.}] There exists a constant $\alpha$ such that
    \[(\bm\sigma_h,\bm \sigma_h)\geq \alpha \|\bm\sigma_h \|_{H(\div;\Omega)}^2\text{ on }\mathbb K_h,\]
     where $\mathbb K_h = \{\bm\tau_h\in V_h^{k}:(\div\bm\tau_h, v_h)=0\text{ for any }v_h\in W_h^{k-1}\}$.
	\item[\textbf{H2.}] There exists an $L^2$-orthogonal projection $P_h:L^2(\Omega)\rightarrow W_h^{k-1}$ satisfying
		\begin{align}
		\|\phi-P_h \phi\|\leq \omega_0(h)\|\grad\phi\| \quad \forall\phi\in H^1(\Omega).
	\end{align}
	\item[\textbf{H3.}] There exists a Fortin interpolation $\Pi_h:V\rightarrow V_h^k$ such that for any $\bm\tau\in V$,
	\begin{align}
		&\div \Pi_h \bm\tau = P_h\div\bm\tau,\label{com-pih}\\
		&\|\bm\tau-\Pi_h \bm\tau\|\leq \omega_1(h)\|\bm\tau\|_V.\label{approxi}
	\end{align}
\end{itemize}
	Here $\omega_i(h)>0$ and $\lim_{h\rightarrow 0^+} \omega_i(h)=0, i=0,1$.
\end{assumption}

According to \cite[Theorem 3.2]{boffi1999computational}, under Assumption \ref{assump},
\[\|T-T_h\|\rightarrow 0,\]
which implies the convergence of the discrete eigenvalues computed from \eqref{fem1} and \eqref{fem2}. 
Assumptions \textbf{H1} and \textbf{H2} are easy to verify. To construct a Fortin interpolation that satisfies \textbf{H3}, in the following section, we fit $V_h^{k}$ and $W_h^{k-1}$ into a finite element complex.

\section{A Finite Element De Rham Complex on $\mathcal Q_h$}
For $k=2,3$, we will fit finite element spaces $V_h^k$ and $W_h^{k-1}$ into a finite element Stokes complexes:
\begin{equation}\label{H2complex2-2D-fe}
\begin{tikzcd}
\!0 \arrow{r} &\mathbb{R} \arrow{r}{\subset} &\! \Sigma_h^{k+1} \! \arrow{r}{\curl} &  V_h^{k} \arrow{r}{\div} &W_h^{k-1}  \arrow{r}&0.
 \end{tikzcd}
\end{equation}

To this end, we first construct local function spaces $\Sigma_h^{k+1}(Q) $, $V_h^{k} (Q)$, and $W^{k-1}_h(Q)$  such that they form a local de Rham complex:
\begin{equation}\label{H2complex2-2D-fe-local}
\begin{tikzcd}
\!0 \arrow{r} &\mathbb{R} \arrow{r}{\subset} &\! \Sigma_h^{k+1}(Q) \! \arrow{r}{\curl} &  V_h^{k} (Q)\arrow{r}{\div} &W^{k-1}_h(Q)  \arrow{r}&0,
 \end{tikzcd}
\end{equation}
where
\[\Sigma_h^{k+1}(Q) = \{ u \in H^2(Q): u|_T \in P_{k+1}(T) \text{ for any }T\in \mathcal T_h^Q\},\]
\[V_h^{k}(Q) = \{ u \in H^1(Q): u|_T \in P_{k}(T) \text{ for any }T\in \mathcal T_h^Q\},\ \text{and }\]
\[W_h^{k-1}(Q) = \div V_h^{k}(Q).\]

We will present the DOFs for each space in the following paragraphs. The DOFs for each space are also shown in Figures \ref{localseq} and \ref{localseq2}. 

\begin{figure}[h!]
\begin{center} 
\setlength{\unitlength}{1.2cm}
\begin{picture}(5,2)(1,-0)
\put(0,0){
\begin{picture}(2,0)
\put(-3, 0){\line(1,4){0.5}} 
\put(-3, 0){\line(5,4){2.5}} 
\put(-2.5, 2){\line(1,0){2}} 
\put(-2.5, 2){\line(5,-4){2.5}}
\put(-0.5, 2){\line(1,-4){0.5}}
\put(-3,0){\line(2,0){3}}
\put(-1.5, 2){\line(0,1){0.3}} 
\put(-1.5, 0){\line(0,-1){0.3}} 
\put(-2.75, 1){\line(-4,1){0.3}} 
\put(-0.25, 1){\line(4,1){0.3}} 
\put(-3,0){\circle*{0.1}}
\put(-3,0){\circle{0.3}}
\put(-2.5,2){\circle*{0.1}}
\put(-2.5,2){\circle{0.3}}
\put(-0.5,2){\circle*{0.1}}
\put(-0.5,2){\circle{0.3}}
\put(0,0){\circle*{0.1}}
\put(0,0){\circle{0.3}}

\end{picture}
}

\put(0.5, 1){\vector(1, 0){1}}
\put(0.68, 1.15){$\curl$}

\put(5,0){
\put(-3, 0){\line(1,4){0.5}} 
\put(-3, 0){\line(5,4){2.5}} 
\put(-2.5, 2){\line(1,0){2}} 
\put(-2.5, 2){\line(5,-4){2.5}}
\put(-0.5, 2){\line(1,-4){0.5}}
\put(-3,0){\line(2,0){3}}

\put(-1.55, 2){\circle*{0.1}}
\put(-1.55, 0){\circle*{0.1}}
\put(-2.7, 1){\circle*{0.1}}
\put(-0.2, 1){\circle*{0.1}}

\put(-2.95,0){\circle*{0.1}}
\put(-2.45,2){\circle*{0.1}}
\put(-0.45,2){\circle*{0.1}}
\put(-0.05,0){\circle*{0.1}}
\put(-1.55,1.2){\circle*{0.1}}
\put(-2.05,1.6){\circle*{0.1}}
\put(-0.7,0.6){\circle*{0.1}}
\put(-2.3,0.6){\circle*{0.1}}
\put(-1.05,1.6){\circle*{0.1}}

\put(-1.45, 2.){\circle*{0.1}}
\put(-1.45, -0){\circle*{0.1}}
\put(-2.8, 1){\circle*{0.1}}
\put(-0.3, 1){\circle*{0.1}}
\put(-3.05,0){\circle*{0.1}}
\put(-2.54,2){\circle*{0.1}}
\put(-0.55,2){\circle*{0.1}}
\put(0.05,0){\circle*{0.1}}
\put(-1.45,1.2){\circle*{0.1}}
\put(-1.95,1.6){\circle*{0.1}}
\put(-0.8,0.6){\circle*{0.1}}
\put(-2.2,0.6){\circle*{0.1}}
\put(-0.95,1.6){\circle*{0.1}}
}

\put(5.5, 1){\vector(1, 0){1}}
\put(5.75, 1.1){{$\div$}}
\put(10,0){
\begin{picture}(2,2)
\put(-3, 0){\line(1,4){0.5}} 
\put(-3, 0){\line(5,4){2.5}} 
\put(-2.5, 2){\line(1,0){2}} 
\put(-2.5, 2){\line(5,-4){2.5}}
\put(-0.5, 2){\line(1,-4){0.5}}
\put(-3,0){\line(2,0){3}}
\put(-1.8,0.7){+11}
\end{picture}
}
\end{picture}
\end{center}
\caption{The finite element complex \eqref{H2complex2-2D-fe-local} with $k=2$.}
\label{localseq}
\end{figure}

\begin{figure}[h!]
\begin{center} 
\setlength{\unitlength}{1.2cm}
\begin{picture}(5,2)(1,-0)
\put(0,0){
\begin{picture}(2,0)
\put(-3, 0){\line(1,4){0.5}} 
\put(-3, 0){\line(5,4){2.5}} 
\put(-2.5, 2){\line(1,0){2}} 
\put(-2.5, 2){\line(5,-4){2.5}}
\put(-0.5, 2){\line(1,-4){0.5}}
\put(-3,0){\line(2,0){3}}
\put(-3,0){\circle*{0.1}}
\put(-3,0){\circle{0.3}}
\put(-2.5,2){\circle*{0.1}}
\put(-2.5,2){\circle{0.3}}
\put(-0.5,2){\circle*{0.1}}
\put(-0.5,2){\circle{0.3}}
\put(0,0){\circle*{0.1}}
\put(0,0){\circle{0.3}}
\put(-2.67, 1.34){\line(-4,1){0.3}} 
\put(-2.84, 0.67){\line(-4,1){0.3}} 
\put(-1.17,2){\line(0,1){0.3}} 
\put(-1.84,2){\line(0,1){0.3}} 
\put(-1.17, 0){\line(0,-1){0.3}} 
\put(-1.84, 0){\line(0,-1){0.3}} 
\put(-0.17, 0.67){\line(4,1){0.3}} 
\put(-0.34, 1.34){\line(4,1){0.3}} 
\put(-1.5, 2){\circle*{0.1}}
\put(-1.5, 0){\circle*{0.1}} 
\put(-2.75, 1){\circle*{0.1}}
\put(-0.25, 1){\circle*{0.1}}
\put(-1.8,0.7){+4}
\end{picture}
}

\put(0.5, 1){\vector(1, 0){1}}
\put(0.68, 1.15){$\curl$}

\put(5,0){
\put(-3, 0){\line(1,4){0.5}} 
\put(-3, 0){\line(5,4){2.5}} 
\put(-2.5, 2){\line(1,0){2}} 
\put(-2.5, 2){\line(5,-4){2.5}}
\put(-0.5, 2){\line(1,-4){0.5}}
\put(-3,0){\line(2,0){3}}


\put(-2.95,0){\circle*{0.1}}
\put(-2.45,2){\circle*{0.1}}
\put(-0.45,2){\circle*{0.1}}
\put(-0.05,0){\circle*{0.1}}
\put(-1.55,1.2){\circle*{0.1}}

\put(-3.05,0){\circle*{0.1}}
\put(-2.54,2){\circle*{0.1}}
\put(-0.55,2){\circle*{0.1}}
\put(0.05,0){\circle*{0.1}}
\put(-1.45,1.2){\circle*{0.1}}

\put(-2.62, 1.34){\circle*{0.1}} 
\put(-2.79, 0.67){\circle*{0.1}}
\put(-1.12,2){\circle*{0.1}}
\put(-1.79,2){\circle*{0.1}}
\put(-1.12, 0){\circle*{0.1}}
\put(-1.79, 0){\circle*{0.1}}
\put(-0.12, 0.67){\circle*{0.1}}
\put(-0.29, 1.34){\circle*{0.1}}

\put(-2.72, 1.34){\circle*{0.1}} 
\put(-2.89, 0.67){\circle*{0.1}}
\put(-1.22,2){\circle*{0.1}}
\put(-1.89,2){\circle*{0.1}}
\put(-1.22, 0){\circle*{0.1}}
\put(-1.89, 0){\circle*{0.1}}
\put(-0.22, 0.67){\circle*{0.1}}
\put(-0.39, 1.34){\circle*{0.1}}

\put(-1.78,1.46){\circle*{0.1}}
\put(-1.88,1.46){\circle*{0.1}}
\put(-1.12,1.46){\circle*{0.1}}
\put(-1.22,1.46){\circle*{0.1}}
\put(-2.11,1.72){\circle*{0.1}}
\put(-2.21,1.72){\circle*{0.1}}
\put(-0.88,1.72){\circle*{0.1}}
\put(-0.78,1.72){\circle*{0.1}}

\put(-2.55,0.4){\circle*{0.1}}
\put(-0.55,0.4){\circle*{0.1}}
\put(-2.05,0.8){\circle*{0.1}}
\put(-1.05,0.8){\circle*{0.1}}

\put(-2.45,0.4){\circle*{0.1}}
\put(-0.45,0.4){\circle*{0.1}}
\put(-1.95,0.8){\circle*{0.1}}
\put(-0.95,0.8){\circle*{0.1}}

\put(-2.5,1.0){+2}
\put(-1,1.0){+2}
\put(-1.75,0.5){+2}
\put(-1.75,1.5){+2}
}

\put(5.5, 1){\vector(1, 0){1}}
\put(5.75, 1.1){{$\div$}}
\put(10,0){
\begin{picture}(2,2)
\put(-3, 0){\line(1,4){0.5}} 
\put(-3, 0){\line(5,4){2.5}} 
\put(-2.5, 2){\line(1,0){2}} 
\put(-2.5, 2){\line(5,-4){2.5}}
\put(-0.5, 2){\line(1,-4){0.5}}
\put(-3,0){\line(2,0){3}}
\put(-1.8,0.7){+23}
\end{picture}
}
\end{picture}
\end{center}
\caption{The finite element complex \eqref{H2complex2-2D-fe-local} with $k=3$.}
\label{localseq2}
\end{figure}

%
%

\paragraph{\textbf{Space $\Sigma_h^{k+1}(Q)$}}
A function $ u \in \Sigma_h^{k+1}(Q)$ can be determined by 
\begin{itemize}
	\item $u(v),\ \partial_{x}u(v),\ \partial_{y}u(v)$ for all  $v\in\mathcal V_h(Q)$;
	\item $\int_{e}\partial_{n_{e}} u q\d s$ for all $q=\hat q\circ F_Q^{-1}$ with $\hat q\in P_{k-2}(\hat e)$ and $e\in\mathcal E_h(Q)$;
	\item $\frac{1}{|e|}\int_{e} u q\d s$ for all $q=\hat q\circ F_Q^{-1}$ with $\hat q\in P_{k-3}(\hat e)$  and $e\in\mathcal E_h(Q)$;
	\item $\frac{1}{|T|}\int_{T} u q\d A$ for all $q=\hat q\circ F_Q^{-1}$ with $\hat q\in P_{k-4}(\hat T)$ and $T\in\mathcal T_h^Q$.
	\end{itemize}

\begin{lemma}
	The above DOFs for space $\Sigma_h^{k+1}(Q)$ is unisolvent and $H^2$-conforming.
	
\end{lemma}
\begin{proof}
	The unisolvence and $H^2$-conformity of $\Sigma_h^{3}(Q)$ with the boundary DOFs $\int_{e_i}\partial_{n_{i}} u\d s$ replaced by $\partial_{n_{i}} u(M_{e_i})$, where $M_{e_i}$ is the midpoint of $e_i$, can be found in \cite[Theorem 6.19]{lai2007spline}. The proof is based on the Bernstein-B\'ezier technique. 
		The unisolvence and $H^2$-conformity of the above DOFs for $k=2,3$ can be proved similarly.
\end{proof}

\paragraph{\textbf{Space $V_h^{k}(Q)$}}
Space $V_h^{k}(Q)$ is the Lagrange finite element space on $\mathcal T_h^Q$.
A function $\bm u \in V_h^{k}(Q)$ can be determined by 
\begin{itemize}
	\item $\bm u(v)$ for all  $v\in\mathcal V_h(Q)$;
	\item $\int_{e}\bm u \cdot \bm q\d s \text{ for all } \bm q\in [P_{k-2}(e)]^2$ and   $e\in\mathcal E_h(Q)$;
	\item $\int_{Q}\div\bm u \cdot q\d A \text{ for all } q\in W_h^{k-1}(Q)\cap L_0^2(Q)$;
	\item $\int_{Q}\bm u \cdot \curl q\d A \text{ for all } q\in \Sigma_h^{k+1}(Q)\cap H_0^2(Q)$.
\end{itemize}
\begin{lemma}
	The above DOFs for $V_h^{k}(Q)$ are unisolvent and $H^1$-conforming. 
\end{lemma}
\begin{proof}
	The $H^1$-conformity can be easily proved by the vertex and edge DOFs. Now to show the unisolvence, we suppose that the above DOFs vanish on $\bm u\in V_h^{k}(Q)$ and prove that $\bm u=0.$ From integration by parts and the edge DOFs, we have
	\[\int_{Q}\div\bm u \d A = \int_{\partial Q}\bm u \cdot \bm n_{Q}\d s = 0, \]
	which together with the interior DOFs $\int_{Q}\div\bm u \cdot q\d A \text{ for any } q\in W_h^{k-1}(Q)\cap L_0^2(Q)$ yields
	\[\div\bm u  =0 \text{ in }Q.\]
	Then for $T\in\mathcal T_h^Q$, there exists a function $\phi_T\in P_{k+1}(T)$ such that $\bm u|_T= \curl \phi_T$. Since $\bm u =0$ on $\partial Q$, we have the tangential derivative of $\phi_T$ along  $\partial T\cap\partial Q$ is 0. Without loss of generality, we can choose $\phi_T$ such that $\phi_T =0$ on $\partial T\cap\partial Q$. Define a function $\phi$ such that $\phi|_T=\phi_T$, then $\phi = 0$ on $\partial Q$. Moreover, we have $\phi\in H_0^2(Q)$  since $\bm u\in \bm H_0^1(Q)$. Take $q=\phi$ in the last category of DOFs for $V_h^{k}(Q)$, we have $\curl\phi=0$ and hence $\bm u=0$.
\end{proof}

\begin{remark} Since space $V_h^{k}(Q)$ is the Lagrange finite element space on $\mathcal T_h^Q$, a function $\bm u \in V_h^{k}(Q)$ can also be determined by
\begin{itemize}
	\item $\bm u(v)$ for all  $v\in\mathcal V_h(Q)$;
	\item $\int_{e}\bm u \cdot \bm q d s$, for all $\bm q\in [P_{k-2}(e)]^2$ and $e\in\mathcal E_h(Q)$;
	\item $\bm u(v_0)$ ($v_0$ is the intersection of the two diagonals of $Q$);
	\item $\int_{e_i^0}\bm u\cdot\bm q\d s$, for all $\bm q\in [P_{k-2}(e_i^0)]^2$, $ i = 1, 2, 3, 4$ ($e_i^0$, $i=1,2,3,4$ are the four interior edges in $Q$);
	\item $\int_{T} \bm u \cdot\bm q\d A$ for all $\bm q\in P_{k-3}(T)$ and $T\in\mathcal T_h^Q$.
\end{itemize}
	
\end{remark}
\begin{lemma}\label{exactnessQ}
	The sequence \eqref{H2complex2-2D-fe-local} is a complex and is exact. 
\end{lemma}
\begin{proof}
	From the construction, we have that the sequence is a complex. It suffices to show the exactness. The exactness at $W_h^{k-1}(Q)$ and $\Sigma_h^{k+1}(Q)$ is trivial. To show the exactness at $V_h^k(Q)$, we suppose $\bm u\in V_h^k(Q)$ satisfying $\div\bm u=0$ and show that $\bm u=\curl \phi$ for some $\phi$ in $\Sigma_h^{k+1}(Q)$. Since $V_h^k(Q)$ is a subspace of the Brezzi-Douglas-Marini finite element space, we have, from the exactness of the de Rham complex, 
	\[\bm u=\curl \phi,\] 
	for some $\phi$ in the Lagrange finite element space of order $k+1$. Due to the fact that $V_h^k(Q)\subset \bm H^1(Q)$, we can conclude that
	$\phi\in \Sigma_h^{k+1}(Q)$.
\end{proof}

\paragraph{\textbf{Characterization of Space $W_h^{k-1}(Q)$}}
To use the mixed formulation \eqref{fem1}, we now give a characterization of $W_h^{k-1}(Q)$. Define 
\begin{align*}
	P_h^{k-1}(Q) &= \{ u \in L^2(Q): u|_T \in P_{k-1}(T) \text{ for any }T\in \mathcal T_h^Q\},
\end{align*}
then $W_h^{k-1}(Q)$ is a subspace of $P_h^{k-1}(Q)$.

From Lemma \ref{exactnessQ}, we have
\[\dim W_h^{k-1}(Q)=1-\dim \Sigma_h^{k+1}(Q)+\dim V_h^k(Q) = \dim P_h^{k-1}(Q) -1.\]
Therefore, to characterize $W_h^{k-1}(Q)$, we need to identify a function which is in $P_h^{k-1}(Q)$ but not in $W_h^{k-1}(Q)$.

We first characterize $W_h^{k-1}(\hat Q)$. According to \cite[Proposition 2.1]{scott1985norm}, for $\hat u\in W_h^{k-1}(\hat Q)$, the following condition holds
\[\hat u|_{\hat T_1}(0,0)+\hat u|_{\hat T_3}(0,0) = \hat u|_{\hat T_2}(0,0)+\hat u|_{\hat T_4}(0,0), \]
which imposes a restriction on the constant term of $\hat u$. Therefore, the piecewise constant function shown in Figure \ref{PQ} does not belong to $W_h^{k-1}(\hat Q)$.

%
Define
\[\widetilde P_h^{k-1}(Q) = \{ u \in L^2(Q): u|_T \in \widetilde P_{k-1}(T) \text{ for any }T\in \mathcal T_h^Q\}.\]
We can deduce that 
the space $W_h^{1}(\hat Q)$ is spanned by the functions in Figure \ref{WH1} and in $\widetilde P_h^{1}(\hat Q)$, and the space $W_h^{2}(\hat Q)$ is spanned by the functions in Figure \ref{WH1} and in $\widetilde P_h^1(\hat Q)+\widetilde P_h^2(\hat Q)$. 

We then characterize $W_h^{k-1}(Q)$ as
\[W_h^{k-1}(Q)=\{q:  q=\hat q\circ F_Q^{-1}\text{ for }\hat q \in  W_h^{k-1}(\hat Q)\}.\]

\begin{figure}[h!]
\vspace{1cm}
\begin{center}
\setlength{\unitlength}{1.2cm}
\begin{picture}(5,2)(1.5,-0)
\put(5.3,0){
\begin{picture}(2,0)
\put(-3, 0){\line(0,3){3}} 
\put(-3, 3){\line(3,0){3}} 
\put(-3,0){\line(3,0){3}}
\put(0,0){\line(0,3){3}}
\put(-3,0){\line(1,1){3}}
\put(-3,3){\line(1,-1){3}}
\put(-1.8,0.4){$-1$}
\put(-1.8,2.3){$-1$}
\put(-2.5,1.4){$1$}
\put(-0.7,1.4){$1$}
\end{picture}}
\end{picture}
\caption{A piecewise constant function not in $W_h^{k-1}(\hat Q)$} \label{PQ}
\end{center}
\end{figure}
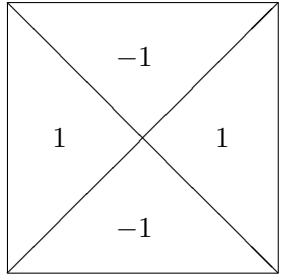

\begin{figure}[h!]
\vspace{1.5cm}
\begin{center} 
\setlength{\unitlength}{1.2cm}

\begin{picture}(5,2)(-0.3,-0)
\put(0,0){
\begin{picture}(2,0)
\put(-3, 0){\line(0,2){3}} 
\put(-3, 3){\line(2,0){3}} 
\put(-3,0){\line(2,0){3}}
\put(-0,0){\line(0,2){3}}
\put(-3,0){\line(1,1){3}}
\put(-3,3){\line(1,-1){3}}
\put(-1.55,0.8){$0$}
\put(-1.55,2){$0$}
\put(-1,1.5){$-1$}
\put(-2.3,1.5){$1$}
\end{picture}}
\put(3.5,0){
\begin{picture}(2,0)
\put(-3, 0){\line(0,2){3}} 
\put(-3, 3){\line(2,0){3}} 
\put(-3,0){\line(2,0){3}}
\put(-0,0){\line(0,2){3}}
\put(-3,0){\line(1,1){3}}
\put(-3,3){\line(1,-1){3}}
\put(-1.7,0.8){$-1$}
\put(-1.55,2){$1$}
\put(-1,1.5){$0$}
\put(-2.3,1.5){$0$}
\end{picture}}
\put(7,0){
\begin{picture}(2,0)
\put(-3, 0){\line(0,2){3}} 
\put(-3, 3){\line(2,0){3}} 
\put(-3,0){\line(2,0){3}}
\put(-0,0){\line(0,2){3}}
\put(-3,0){\line(1,1){3}}
\put(-3,3){\line(1,-1){3}}
\put(-1.55,0.8){$1$}
\put(-1.55,2){$1$}
\put(-1,1.5){$1$}
\put(-2.3,1.5){$1$}\end{picture}}
\end{picture}
\caption{Piecewise constant functions in $W_h^{k-1}(\hat Q)$.}\label{WH1}
\end{center} 

\end{figure}
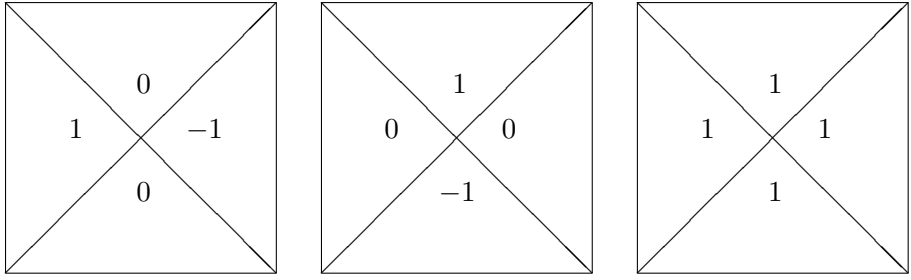

We glue together $\Sigma_h^{k+1}(Q)$ for $Q\in\mathcal Q_h$ by the DOFs for $\Sigma_h^{k+1}(Q)$ to define a global finite element space $\Sigma_h^{k+1}$. Similarly, we can define $V_h^k$ and $W_h^{k-1}$. 
\begin{lemma}
	The sequence \eqref{H2complex2-2D-fe} is an exact complex on contractible domains.
\end{lemma}
\begin{proof}
	From the construction, we have the sequence is a complex. It suffies to show the exactness. The exactness at $W_h^{k-1}$ and $\Sigma_h^{k+1}$ is trivial. To show the exactness at $V_h^k$, we use dimension count.  The dimensions of $\Sigma_h^{k+1}$ and $V_h^k$ are respectively
	\begin{align*}
	\begin{split}
		&\dim\Sigma_h^{k+1}=3\mathcal V+(2k-3)\mathcal E+4(k-2)\mathcal Q,\\
		&\dim V_h^k = 2\mathcal V+(2k-2)\mathcal E+\dim W_h^{k-1}-\mathcal Q+4(k-2)\mathcal Q,
	\end{split}
	\end{align*}
	where $\mathcal V$, $\mathcal E$, $\mathcal Q$ are the number of vertices, edges and  convex quadrilaterals in $\mathcal Q_h$.
Then we have 
\[1-\dim \Sigma_h^{k+1}+\dim V_h^k-\dim W_h^{k-1}=1-\mathcal V+\mathcal E-\mathcal Q = 0,\]
from the Euler's formula. 	This completes the proof. 
\end{proof}

\section{The Construction of Fortin Interpolation}\label{sec4}
In this section, we construct a Fortin interpolation $\Pi_h$ that satisfies \textbf{H3} in Assumption \ref{assump}. To this end, we first define
 an interpolation $\Pi^0_h:V \rightarrow V_h^k$ whose restriction on $Q$ is denoted by $\Pi^0_Q$. For $\bm u\in V$, $\Pi_Q^0\bm u$ is uniquely determined by
\begin{itemize}
	\item $\Pi_Q^0\bm u(v) = 0$ for any $v\in \mathcal V_h(Q)$;
    \item $\int_{e}\Pi_Q^0\bm u \cdot \bm \tau_e q\d s = 0 \text{ for any } q\in P_{k-2}(e)\text{ and }e\in \mathcal E_h(Q)$;
	\item $\int_{e}\Pi_Q^0\bm u \cdot \bm n_e q\d s = \int_{e}\bm u \cdot \bm n_e q\d s \text{ for any } q\in P_{k-2}(e)\text{ and }e\in \mathcal E_h(Q)$;
	\item $\int_{Q}\div\Pi_Q^0\bm u \cdot q\d A = \int_{Q}\div\bm u \cdot q\d A\text{ for any } q\in W_h^{k-1}(Q)\cap L_0^2(Q)$;
	\item $\int_{Q}\Pi_Q^0\bm u \cdot \curl q\d A = 0 \text{ for any } q\in \Sigma_h^k(Q)\cap H_0^2(Q)$.
\end{itemize}

Recall $P_h: L^2(\Omega)\rightarrow W_h^{k-1}$ is the $L^2$-orthogonal projection. 
From the non-vanishing edge DOFs, we have for any $Q\in\mathcal Q_h$, 
\[\int_Q \div\Pi_h^0\bm u - P_h\div\bm u\d A= \int_Q \div\Pi_h^0\bm u - \div\bm u \d A =  \int_{\partial Q} (\Pi_h^0\bm u - \bm u )\cdot \bm n_Q\d s =0,\]
which together with the non-vanishing interior DOFs leads to
$\div\Pi_h^0\bm u = P_h\div\bm u$.

We then define $\Pi_h: V\rightarrow V_h^k$ by
\begin{align}
\Pi_h = \Pi_h^0(\I-\bm I_h)+\bm I_h,\label{Pihdef}
\end{align}
where $\bm I_h:\bm H^{1/2+\delta}(\Omega)\rightarrow V_h^k$ is the Scott-Zhang interpolation defined in \cite{scott1990finite}. According to \cite[Theorem 3.3]{Ciarlet+2013+173+180}, for $T\in \mathcal T_h$,
\[h_T^{-1/2-\delta}\|\bm u-\bm I_h\bm u\|_{T}+\|\bm I_h\bm u\|_{1/2+\delta,T}\leq C\|\bm u\|_{1/2+\delta,\omega(T)},\]
where $\omega(T)=\mathop{\cup}\limits_{T'\in\mathcal T_h\atop\bar T'\cap \bar T\neq\emptyset}T'.$

Under the transformation 
\[\bm u\circ F_i = \frac{B_i}{|B_i|}\hat{\bm u},\]
we have
\begin{itemize}
	\item $\int_{e}\bm u\cdot\bm n_e q\d s = \int_{\hat e}\hat{\bm u} \cdot \hat{\bm n}_e q\d \hat{s}$,
	\item $\int_{Q}\div\bm u \cdot q\d A=\int_{\hat Q}\widehat{\div}\hat{\bm u} \cdot \hat{q}\d \hat A$,
\end{itemize}
which implies $\widehat{\Pi_Q^0\bm u}=\Pi_{\hat Q}^0\hat{\bm u}$.
By scaling argument, we obtain the following stability estimate of $\Pi_h^0$,
\begin{align}
	\|\Pi_Q^0\bm u\|_Q^2 &\leq C\|\widehat{\Pi_Q^0\bm u}\|_{\hat Q}^2 =C\|\Pi_{\hat Q}^0\hat{\bm u}\|_{\hat Q}^2 \leq C\left(\|\hat{\bm u}\|^2_{1/2+\delta,\hat Q}+\|\widehat \div\hat{\bm u}\|_{\hat Q}^2\right)\nonumber\\
	&\leq C\left(\|{\bm u}\|^2_{Q}+h^{1+2\delta}|{\bm u}|^2_{1/2+\delta, Q}+h^2\|\div{\bm u}\|_{Q}^2\right).\label{stability}
\end{align}

\begin{theorem}\label{theorem1}
	The interpolation  $\Pi_h$ defined by \eqref{Pihdef} satisfies  Assumption \ref{assump}. 
\end{theorem}
\begin{proof}
We first prove \eqref{com-pih}. For any $q_h\in W_h^{k-1}$, we have
\begin{align*}
(\div\Pi_h\bm u-P_h\div\bm u, q_h)& = (\div\Pi_h^0(\I-\bm I_h)\bm u+\div\bm I_h\bm u-P_h\div\bm u, q_h)\\
&= (P_h\div(\I-\bm I_h)\bm u+\div\bm I_h\bm u-P_h\div\bm u, q_h)=0,
\end{align*}
which implies \eqref{com-pih}.

We now prove \eqref{approxi}. 
By the definition of $\Pi_h$ and the stability estimate \eqref{stability} of $\Pi_h^0$,
\begin{align}\label{Pih}
\begin{split}
	&\|(\I-\Pi_h)\bm u\| ^2= \|(\I-\Pi_h^0)(\I-\bm I_h)\bm u\|^2 \\
&\leq C\left(\|(\I-\bm I_h)\bm u\|^2+h^{1+2\delta}|(\I-\bm I_h)\bm u|^2_{1/2+\delta}+h^2\|\div(\I-\bm I_h)\bm u\|^2\right).
\end{split}
\end{align}
To estimate $\|\div(\I-\bm I_h)\bm u\|$, we introduce the canonical interpolation $\bm r_T:\bm H^{1/2+\delta}(T)\rightarrow \bm P_0(T)\oplus \bm x P_0(T)$ for $T\in\mathcal T_h$ such that for $\bm u \in \bm H^{1/2+\delta}(T)$,
\[\int_e\bm r_T\bm u\cdot \bm n_e\d s = \int_e\bm u\cdot \bm n_e\d s \quad\text{ for any }e\subset{\partial T}.\]
Assume $\pi_T:L^2(T)\rightarrow P_0(T)$ is the $L^2$-orthogonal projection.  The interpolation $\bm r_T$ satisfies \cite{Monk2003}
\begin{align}
	&\div\bm r_T \bm u =\pi_T\div\bm u,\label{com-rh}\\
	&\|\bm u-\bm r_h \bm u\|_T\leq Ch_T^{1/2+\delta}\|\bm u\|_{1/2+\delta,T}.
\end{align}
Now we estimate  $\|\div(\I-\bm I_h)\bm u\|_T$ for $T\in\mathcal T_h$,
\begin{align*}
	\|\div(\I-\bm I_h)\bm u\|_T&\leq  \|\div(\I-\bm r_T)\bm u\|_T+\|\div(\bm r_T-\bm I_h)\bm u\|_T\\
	&\leq \|\div \bm u-\pi_T\div\bm u\|_T+Ch_T^{-1}\|(\bm r_T-\bm I_h)\bm u\|_T\quad\text{\big(\eqref{com-rh} and inverse inequality\big)}\\
	&\leq C\|\div \bm u\|_T+Ch_T^{-1}\|\bm r_T\bm u-\bm u\|_T+Ch_T^{-1}\|\bm u-\bm I_h\bm u\|_T\\
	&\leq C\|\div \bm u\|_T+Ch_T^{-1/2+\delta}\|\bm u\|_{1/2+\delta,T}+Ch_T^{-1}\|\bm u-\bm I_h\bm u\|_T,
\end{align*}
which, plugged into \eqref{Pih}, leads to 
\begin{align*}
\|(\I-\Pi_h)\bm u\| ^2 &\leq C\left(\|(\I-\bm I_h)\bm u\|^2+h^{1+2\delta}|(\I-\bm I_h)\bm u|^2_{1/2+\delta}+h^2\|\div\bm u\|^2+Ch^{1+2\delta}\|\bm u\|_{1/2+\delta}^2\right)\\
&\leq C\left(h^{1+2\delta}\|\bm u\|_{1/2+\delta}^2+h^2\|\div\bm u\|^2\right)\leq Ch^{1+2\delta}\|\bm u\|_V^2.
\end{align*}

\end{proof}

\begin{theorem}\label{Vbound}
	In addition to  Assumption \ref{assump}, the interpolation  $\Pi_h$ defined by \eqref{Pihdef} also satisfies
	\[\|\Pi_h\bm u\|_1\leq C\|\bm u\|_1.\]
Then $\Pi_h$ defines a $V$-bounded cochain projection for the complex \eqref{H2complex2-2D}. \end{theorem}
\begin{proof}
	By Theorem \ref{theorem1}, we have
	\[\|\Pi_h\bm u\|\leq C\|\bm u\|_1.\]
	It remains to show that
	\[\|\grad\Pi_h\bm u\|\leq C\|\bm u\|_1.\]
	According to the definition of $\Pi_h$, the inverse inequality, \eqref{stability}, and the approximation property $\bm I_h$, we have
	\begin{align*}
		\|\grad\Pi_h\bm u\|^2 = \ &\|\grad\Pi_h^0(\bm u-\bm I_h\bm u)+\grad\bm I_h\bm u\|^2 \\
		\leq \ & Ch^{-2}\|\Pi_h^0(\bm u-\bm I_h\bm u)\|^2+\|\grad\bm I_h\bm u\|^2\\
		\leq \  & C\left(\|{\bm u-\bm I_h\bm u}\|^2+h^{2}|{\bm u-\bm I_h\bm u}|^2_{1}+h^2\|\div(\bm u-\bm I_h\bm u)\|^2+|\bm I_h\bm u|_1^2\right)\\
		\leq \ &C\left(\|{\bm u-\bm I_h\bm u}\|^2+h^{2}|\bm u|_1^2+|\bm I_h\bm u|_1^2\right)\leq C\|\bm u\|_1^2.
	\end{align*}
\end{proof}

\section{Bounded Commuting Projections}
In this section, we will construct bounded commuting projections for the complexes \eqref{H2complex2-2D}. To be specific, we will construct $\pi_h$ by mimicking the Scott-Zhang interpolation \cite{scott1990finite, girault2002hermite} and $\widetilde{\Pi}_h$ by modifying ${\Pi}_h$ in Section 4 such that the following diagram commutes:
\begin{equation}\label{H2complex2-2D-interp}
\begin{tikzcd}
0 \arrow[r] 
& \mathbb{R} \arrow[d,"\mathrm{I}"]\arrow[r,"\subset"]  & H^2(\Omega)\arrow[d,"\pi_h"] \arrow[r,"\curl"]  & \bm H^1(\Omega)\arrow[d,"\widetilde{\Pi}_h" ]\arrow[r,"\div"]  & L^2(\Omega)\arrow[d,"P_h" ]\arrow[r]& 0\\
 0 \arrow[r] 
& \mathbb{R} \arrow[r,"\subset"]  & \Sigma_h^{k+1} \arrow[r,"\curl"]  & V_h^k\arrow[r,"\div"] &\div V_h^k\arrow[r] & 0.
\end{tikzcd}
\end{equation}

\paragraph{\textbf{The construction of}  $\pi_h$}
We denote by $\mathcal N_h =\mathcal N_h^0\cup \mathcal N_h^1$ the set of DOFs for $\Sigma_h^{k+1}$ with 
\[\mathcal N_h^0 = \{N_i, i=1,2,\cdots,2\mathcal V: N_iu = D_ju(v)\text{ for some }  j\in\{1,2\}, \ v\in \mathcal V_h \text{ with } D_1 = \partial_x,\ D_2 = \partial_y \},\]
and \[\mathcal N_h^1=\{N_i,i=2\mathcal V+1,\cdots, \dim\mathcal N_h:N_i\text{ is the DOF in $\mathcal N_h$ but not in }\mathcal N_h^0\}.\] 
For $u\in H^{2+\delta}(\Omega)$, we can define a canonical interpolation $\widetilde \pi_h: H^{2+\delta}(\Omega)\rightarrow \Sigma_h^{k+1}$ by the DOFs in $\mathcal N_h$:
\[ \widetilde \pi_h u = \sum_{i=1}^{2\mathcal V}N_i(u)\phi_{i}+\sum_{2\mathcal V+1}^{\dim \mathcal N_h }N_i(u)\phi_{i}=\sum_{v\in \mathcal V_h}\sum_{j=1}^2D_ju(v)\phi_{j}^v+\sum_{2\mathcal V+1}^{\dim \mathcal N_h }N_i(u)\phi_{i},\]
where $\{\phi_{i}\}_{i=1}^{\dim \mathcal N_h}$ is the nodal basis of $\Sigma_h^{k+1}$ such that $N_i(\phi_j)=\delta_{ij}$ for $N_i\in\mathcal N_h$.  For $i=1,2,\cdots,2\mathcal V$, there exists $j$ and $v$ such that $\phi_i = \phi_j^v$.

By scaling argument, 
\begin{align}\label{scaleofphijv}
\|\phi_j^v\|_{m,Q}\leq Ch_Q^{2-m}, \ j=1,2,\ v\in\mathcal V_h,\quad 
\|\phi_i\|_{m,Q}\leq Ch_Q^{1-m}, \text{ for }i = 2\mathcal V+1,\cdots ,\dim \mathcal N_h.
\end{align}

The DOFs in $\mathcal N_h^0$ are not well defined for $u\in H^2(\Omega)$. 
To define the mapping \(\pi_h: H^{2}(\Omega)\rightarrow \Sigma_h^{k+1}\), we modify DOFs in \(\mathcal N_h^0\) by adopting a similar approach to the one outlined  in \cite[(4.11)]{girault2002hermite}. However, we deviate by setting the parameter \(b_\kappa\) in their equation (4.11) to 1.
 For $v\in \mathcal V_h$, we select any edge $e_v$ such that
\[v\in \overline{e_v},\] 
subject to the restriction $e_v\in \partial\Omega$ if $v\in\partial\Omega$.

For $v\in\mathcal V_h$, according to Riesz's Representation Theorem, there exists a unique function $\psi_v\in P_{k+1}(e_v)$ such that 
\[\int_{e_v}\psi_v(\bm x)w(\bm x)\d \bm x = w(v)\text{ for any }w\in P_{k+1}(e_v).\]
Note that $\psi_v$ depends on the choice of $e_v$.
By construction, we have
\begin{align}\label{dij}
	&\int_{e_v}\psi_v(\bm x)D_j\phi_{i}(\bm x)\d \bm x = 0 \text{ for $i=2\mathcal V+1\cdots,\dim\mathcal N_h$} \text{ and }j=1,2,\\
	&\int_{e_v}\psi_v(\bm x)D_j\phi_{j^\prime}^{v^\prime}(\bm x)\d \bm x = \delta_{(j,v),(j^{\prime},v^{\prime})}.\label{dij2}
\end{align}

Then we define
\[\pi_h u = \sum_{v\in \mathcal V_h}\sum_{j=1,2} \int_{e_v}\psi_v(\bm x)D_ju(\bm x)\d \bm x \phi_{j}^v+\sum_{2\mathcal V+1}^{\dim \mathcal N_h }N_i(u)\phi_{i}.\]
By \eqref{dij}--\eqref{dij2}, we conclude that $\pi_h$ is a projection,
\[\pi_h u = u\ \ \ \ \text{ for any }u\in \Sigma_h. \]

According to \cite[(4.12)]{girault2002hermite}, 
\[\|\psi_v\|_{L^\infty(e_v)} \leq Ch^{-1}.\]
It follows from \cite[(3.6)]{scott1990finite} that
\begin{align}
\|\grad u\|_{L^1(e_v)} &\leq C\sum_{\ell=1}^2h_{T_v}^{\ell-1}|u|_{\ell,T_v}\quad\forall u\in H^2(T_v),\label{du}
\end{align}
where $T_v$ is a triangle in $\mathcal T_h$ that contains $e_v$ as an edge.

\begin{theorem}\label{stablity} Let $u\in H^2(\Omega)$. Then
\[\|\pi_h u\|_{m,Q}\leq C\sum_{\ell=0}^2h^{\ell-m}|u|_{\ell,S_Q},\text{ $m = 0,1,2$},\]
where 
\[S_Q= \text{interior}\big(\cup\{\overline T\in \mathcal T_h:\overline T\cap \overline Q\neq \emptyset\}\big).\]
\end{theorem}
\begin{proof} We denote by $\mathcal N_h(Q)=\mathcal N_h^0(Q)\cup\mathcal N_h^1(Q)$ the set of DOFs for $V_h^{k}(Q).$ Then for $N_i\in\mathcal N_h^1(Q)$, by \eqref{du} and the embedding theorem, we have
\begin{align*}
	|N_i(u)|\leq &C\sum_{e\in \mathcal E_h(Q)}\|\grad u\|_{L^{1}(e)}+C\|u\|_{L^\infty(Q)}
	\leq C\sum_{\ell=1}^2h^{\ell-1}|u|_{\ell,S_Q}+\|\hat u\|_{L^\infty(\hat Q)}\\
	\leq &C\sum_{\ell=1}^2h^{\ell-1}|u|_{\ell,S_Q}+\|\hat u\|_{H^2(\hat Q)}\leq C\sum_{\ell=0}^2h^{\ell-1}|u|_{\ell,S_Q},
\end{align*}
where $C$ is independent on $h^{-1}$. 
	\begin{align*}
		\|\pi_h u\|_{m,Q} &\leq \sum_{N_i\in\mathcal N_h(Q)}|N_i(\pi_h u)| \|\phi_i\|_{m,Q} \\
		&\leq Ch^{2-m}\sum_{N_i\in\mathcal N_h^0(Q)}|N_i(\pi_h u)| \|\hat\phi_i\|_{m,\hat Q}+Ch^{1-m}\sum_{N_i\in\mathcal N_h^1(Q)}|N_i(\pi_h u)| \|\hat\phi_i\|_{m,\hat Q} \\
		& \leq Ch^{2-m}\sum_{v\in \mathcal V_h(K)}\sum_{j=1,2}\Big| \int_{e_v}\psi_v(\bm x)D_ju(\bm x)\d \bm x\Big| +\sum_{N_i\in \mathcal N_h^1(Q)}Ch^{1-m}|N_i(u)|\\
		& \leq Ch^{2-m}\sum_{v\in \mathcal V_h(K)}\sum_{j=1,2}\|\psi_v\|_{L^{\infty}(e_v)}\|D_ju\|_{L^1(e_v)} +\sum_{N_i\in \mathcal N_h^1(Q)}Ch^{1-m}|N_i(u)|\\
		& \leq Ch^{\ell-m}\sum_{v\in \mathcal V_h(K)}\sum_{\ell=1}^2|u|_{\ell,T_v} +Ch^{\ell-m}\sum_{\ell=0}^2|u|_{\ell,S_Q} \leq Ch^{\ell-m}\sum_{\ell=0}^2|u|_{\ell,S_Q} .
	\end{align*}
\end{proof}

\begin{theorem}\label{appx-pi}
	 Let $u\in H^2(\Omega)$. Then
\[\sum_{Q\in \mathcal Q_h}h_Q^{2(m-l)}\|u-\pi_h u\|_{m,Q}^2\leq C|u|^2_{l},\]
with integer $0\leq m\leq l\leq 2$.
\end{theorem}
\begin{proof}
	By following the argument used in \cite[Theorem 4.1]{scott1990finite} and using Theorem \ref{stablity}, we can prove the theorem.
\end{proof}

\paragraph{\textbf{The construction of}  $\widetilde\Pi_h$ }
An $H^{1}$-orthogonal decomposition for
$\bm u\in \bm H^1(\Omega)$ leads to
\[\bm u = \grad w + \bm u^{\perp},\]
where $w\in H^2(\Omega)$ with $\int_\Omega w =0$ and $ \bm u^{\perp}\in \{\bm H^1(\Omega): (\bm u^{\perp},\bm q)+(\grad\bm u^{\perp},\grad\bm q)=0\text{ for any }\bm q\in \grad H^2(\Omega)\}.$ Then
\[\|\grad w\|_1+\|\bm u^{\perp}\|_1\leq \|\bm u\|_1.\]

We define $\widetilde\Pi_h: \bm H^1(\Omega)\rightarrow V_h^k$ by
\[\widetilde\Pi_h \bm u = \grad \pi_h w +\Pi_h\bm u^{\perp}.\]
Clearly, we have the commutative property,
\begin{align}\label{commuting}
\widetilde\Pi_h \grad w = \grad \pi_h w, \quad \curl\widetilde\Pi_h \bm u =  P_h\curl\bm u.	
\end{align}

\begin{theorem}
	Let $\bm u\in \bm H^1(\Omega)$. Then
\begin{align}
\|\bm u-\widetilde \Pi_h\bm  u\|\leq Ch\|\bm u\|_{1},\\
\|\widetilde \Pi_h\bm  u\|_1 \leq C\|\bm u\|_1.\label{vbound2}
\end{align}
\end{theorem}
\begin{proof}By the triangle inequality, Theorem \ref{appx-pi}, Theorem \ref{theorem1}, and the Poincar\'e inequality,
	\begin{align*}
		&\|\bm u-\widetilde \Pi_h\bm  u\|\leq \|\grad w-\grad \pi_h w\|+\|\bm u^{\perp}-\Pi_h\bm u^{\perp}\|\\
		\leq & Ch\|w\|_2 + Ch\|\bm u^{\perp}\|_1\leq Ch\|\grad w\|_1 + Ch\|\bm u^{\perp}\|_1\leq Ch\|\bm u\|_1.
	\end{align*}
	Similarly, by Theorem \ref{Vbound}, we can prove \eqref{vbound2}.
\end{proof}

\begin{theorem}
	Define the cohomology space for the top complex in \eqref{H2complex2-2D}
\begin{align*}
    \mathcal H^0 &:= \ker\big(\curl: H^2(\Omega)\rightarrow \bm H^1(\Omega)\big)\slash \mathbb R,\ \ \mathcal H^2 := \ran\big(\div:\bm H^1(\Omega)\rightarrow L^2(\Omega)\big),\\
	\mathcal H^1 &:= \ker\big(\div:\bm H^1(\Omega)\rightarrow L^2(\Omega)\big)\slash \ran\big(\curl: H^2(\Omega)\rightarrow \bm H^1(\Omega)\big),
\end{align*}
and the cohomology spaces for the bottom complex in \eqref{H2complex2-2D}
\begin{align*}
    \mathcal H_h^0 &:= \ker\big(\curl:\Sigma_h\rightarrow V_h\big)\slash \mathbb R, \ \ \mathcal H_h^2 := \ran\big(\div:V_h\rightarrow \div V_h\big),\\
	\mathcal H_h^1 &:= \ker\big(\div:V_h\rightarrow \div V_h\big)\slash \ran\big(\curl:\Sigma_h\rightarrow V_h\big).
\end{align*}
Then the interpolations $\pi_h$, $\widetilde \Pi_h$, and $P_h$ induce an isomorphism from $\mathcal H^\ell$ onto $\mathcal H_h^\ell$.
\end{theorem}
\begin{proof}
By \eqref{commuting} and Theorems \ref{appx-pi} and \ref{vbound2}, the interpolations $\pi_h$, $\widetilde \Pi_h$, and $P_h$ are bounded commuting projections. 
Then to prove the theorem, according to Theorem 5.1 in \cite{arnold2018finite}, it suffices to show \cite[(5.8)]{arnold2018finite}. We only show \cite[(5.8)]{arnold2018finite} for  $\Pi_h$. It reads
\[\|\bm\varphi-\Pi_h\bm \varphi\|<\|\bm\varphi\|, \ 0\neq \bm\varphi\in \mathfrak H^1,\]
where $\mathfrak H^1$ is the space of harmonic 1-forms with dimension equal to the first Betti number of the domain $\Omega$. 
From Theorem \ref{theorem1}, 
\[\|\bm \varphi-\Pi_h\bm \varphi\|\leq Ch\|\bm\varphi\|_1\leq Ch\|\bm\varphi\|<\|\bm\varphi\|,\text{ for }h \text{ small enough},\]
where we have used the norm equivalence for the finite-dimensional space $\mathfrak H^1$. 
\end{proof}

\section{Numerical Examples}
\subsection{Example 1}
In this example, we compute eigenvalues of \eqref{fem1} on a square domain $\Omega = (0,\pi)\times(0,\pi)$ using quadratic and cubic Lagrange elements on criss-cross mesh. Figure \ref{figure1} shows two criss-cross meshes of $\Omega$ with $h=\pi/8$: a rectangular criss-cross mesh and a quadrilateral criss-cross mesh.

In Tables \ref{table1} and \ref{table3}, we present the first ten eigenvalues and errors computed by quadratic Lagrange elements on rectangular and quadrilateral cirss-cross meshes for a fixed $h=\pi/64$. In Tables \ref{table2} and \ref{table4}, we list the rates of convergence of the first eigenvalue. 

In Table \ref{table5}, we show the first ten eigenvalues and errors computed by cubic Lagrange elements on a rectangular Criss-cross mesh for a fixed $h=\pi/64$. In Table \ref{table6}, we list the rate of convergence of the first eigenvalue.

\begin{figure}[h!]
	\includegraphics[width=0.29\textwidth]{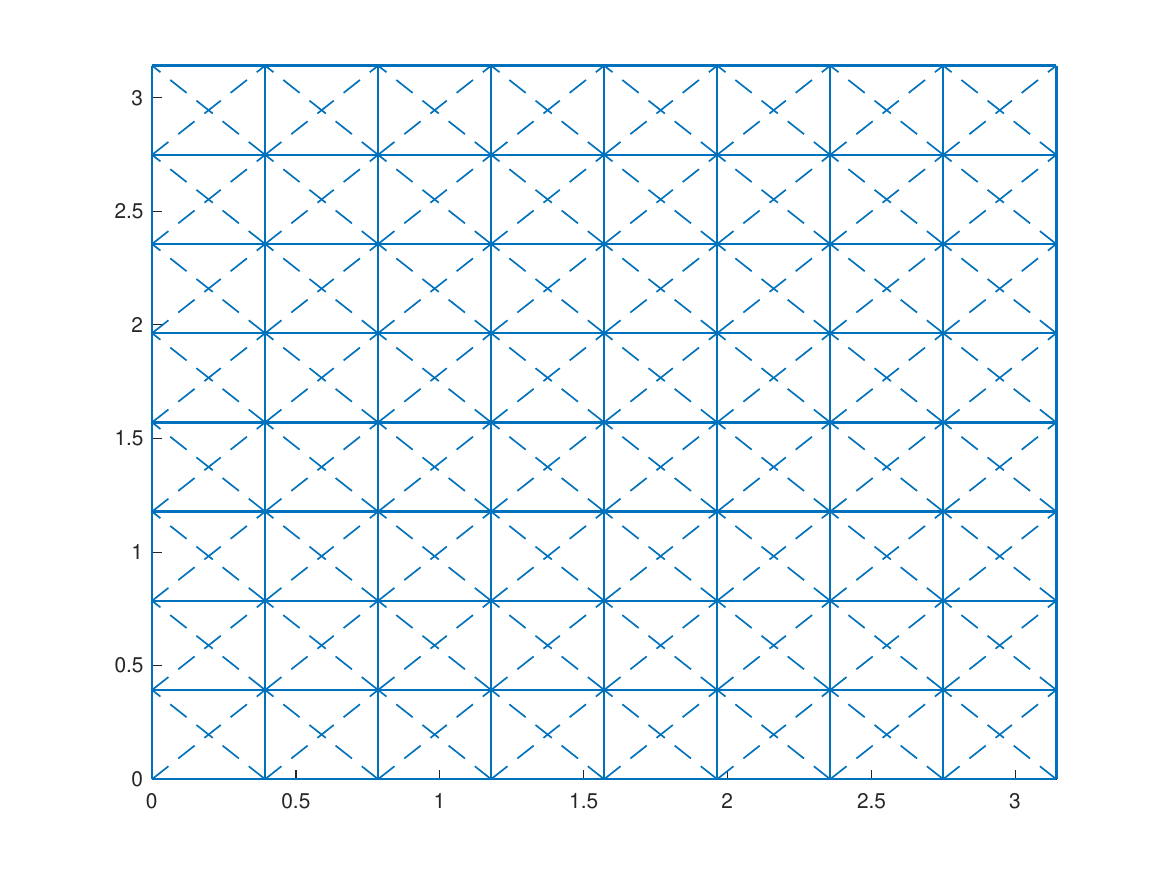}
	\includegraphics[width=0.29\textwidth]{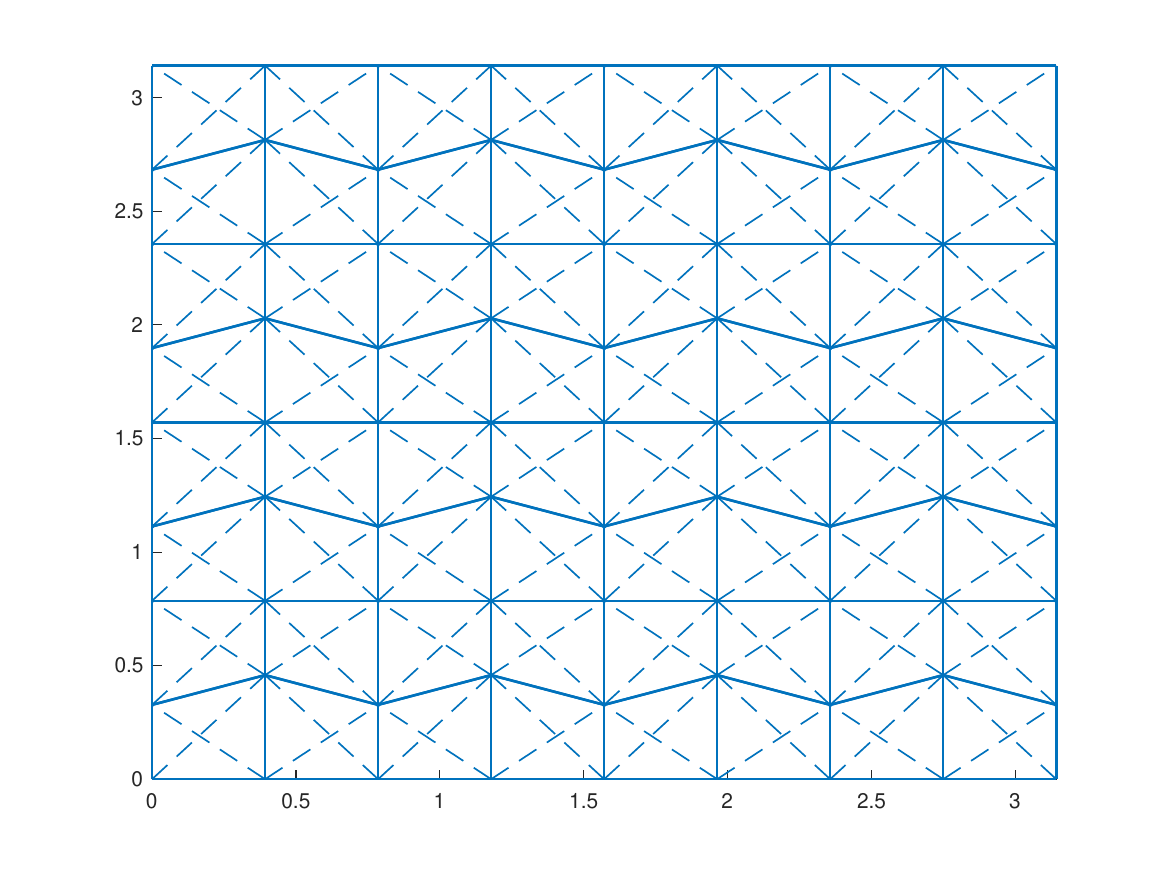}
	\includegraphics[width=0.29\textwidth]{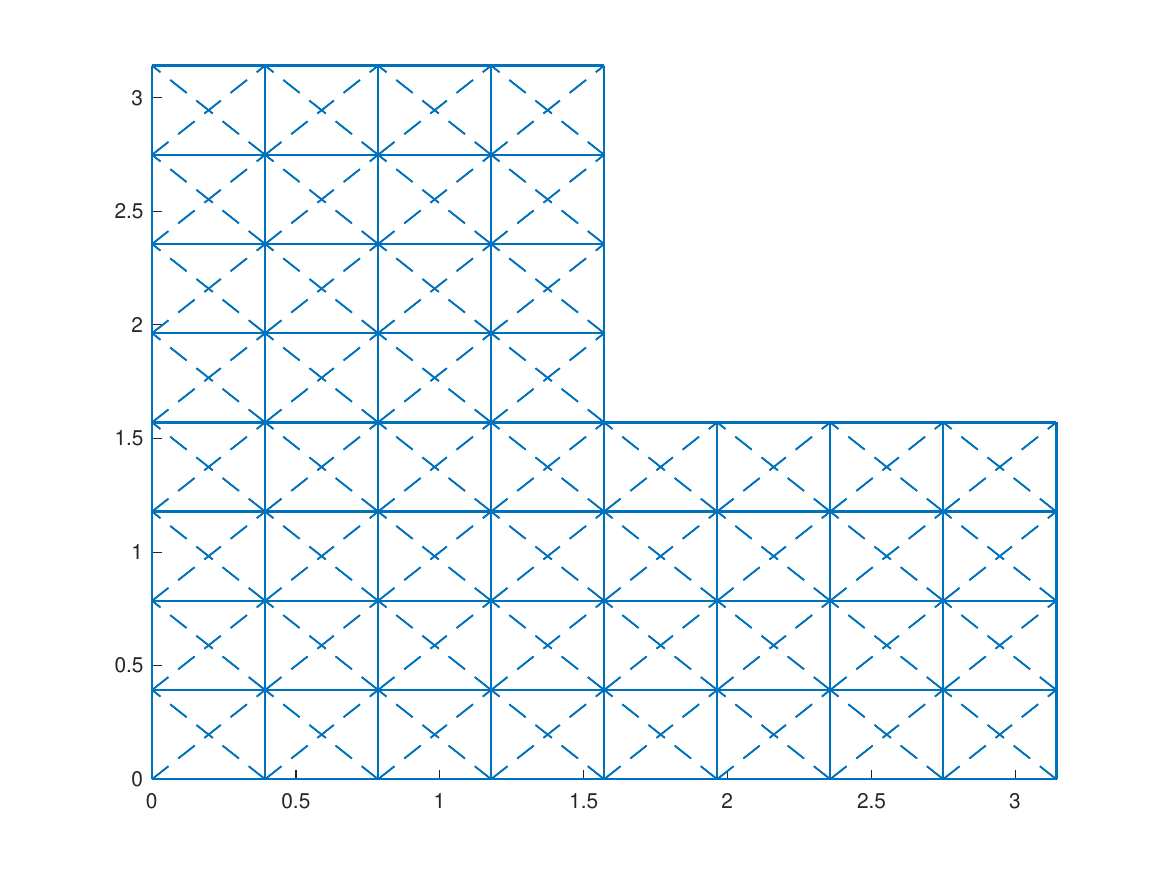}
	\caption{A rectangular criss-cross mesh of $(0,\pi)\times(0,\pi)$  (left), a quadrilateral criss-cross mesh of $(0,\pi)\times(0,\pi)$ (middle), and a rectangular criss-cross mesh of $(0,\pi)\times(0,\pi)\backslash[\frac{\pi}{2},\pi)\times [\frac{\pi}{2},\pi)$ (right)  with $h=\pi/8$}\label{figure1}
\end{figure}
     
\begin{table}[h]
	\caption{Example 1: first ten numerical eigenvalues with $k = 2$ and $h=\pi/64$ on rectangular criss-cross mesh}
	\begin{tabular}{|c|c|c|c|}
	\hline
		$i$&$\lambda^{(i)}$&$\lambda_h^{(i)}$&$\big|\lambda^{(i)}-\lambda_h^{(i)}\big|$\\\hline
      1&2&2.000000009674456&     9.674455903052603e-09\\
      2&5&5.000000169214253 &    1.692142532760954e-07\\
      3&5& 5.000000169215196 &    1.692151956333987e-07\\
      4&8& 8.000000618708425 &    6.187084249376085e-07\\
      5&10&10.00000146560498 &    1.465604976047530e-06\\
      6&10&10.00000146561097 &    1.465610974804576e-06\\
      7&13&13.00000278345120 &    2.783451204636020e-06\\
      8&13&13.00000278345120 &    2.783451204636020e-06\\
     9&17& 17.00000746915364 &    7.469153640471404e-06\\
     10&17& 17.00000746915827 &    7.469158269657328e-06\\
	\hline
	\end{tabular}
\label{table1}
\end{table}
     
\begin{table}[h]
	\caption{Example 1: convergence rate of the first eigenvalue $k = 2$ on rectangular criss-cross mesh}
	\begin{tabular}{|c|c|c|}
	\hline
		$h$&$\big|\lambda^{(1)}-\lambda_h^{(1)}\big|$&rate\\\hline
  $\pi/8$&  3.918771488331529e-05&-\\
   $\pi/16$& 2.468843263603304e-06  & 3.9885\\
  $\pi/ 32$&1.546846171152083e-07  & 3.9964\\
   $\pi/64$&9.674455903052603e-09& 3.9990\\
  	\hline
	\end{tabular}
\label{table2}
\end{table}

\begin{table}[h]
	\caption{Example 1: first ten numerical eigenvalues with $k = 2$ and $h\approx\pi/64$ on quadrilateral criss-cross mesh}
	\begin{tabular}{|c|c|c|c|}
	\hline
		$i$&$\lambda^{(i)}$&$\lambda_h^{(i)}$&$\big|\lambda^{(i)}-\lambda_h^{(i)}\big|$\\\hline
   1&2&2.000000013083459&1.308345876083195e-08\\
   2&5& 5.000000196097877& 1.960978774917521e-07\\
   3&5&5.000000242443800& 2.424438001469298e-07\\
   4&8&8.000000836667605&8.366676045312715e-07\\
   5&10&10.000001598937962& 1.598937961588831e-06\\
  6&10& 10.000002092860292&2.092860292179921e-06\\
  7&13&13.000003434403844&3.434403843982636e-06\\
  8&13&13.000003955633003& 3.955633003371872e-06\\
  9&17&17.000007950604662& 7.950604661743910e-06\\
  10&17&17.000010623894767&  1.062389476658154e-05\\
	\hline
	\end{tabular}
\label{table3}
\end{table}

\begin{table}[h]
	\caption{Example 1: convergence rate of the first eigenvalue $k = 2$ on quadrilateral criss-cross mesh}
	\begin{tabular}{|c|c|c|}
	\hline
		$h$&$\big|\lambda^{(1)}-\lambda_h^{(1)}\big|$&rate\\\hline
  $\pi/8$& 5.287702674161565e-05&-\\
  $\pi/16$&3.337272737269359e-06& 3.9859\\
 $ \pi/32$&   2.091739674803250e-07 & 3.9959\\
  $\pi/64$&  1.308399699695428e-08&3.9988\\
  	\hline
	\end{tabular}
\label{table4}
\end{table}
        
\begin{table}[h]
	\caption{Example 1: first ten numerical eigenvalues with $k = 3$ and $h=\pi/64$ on rectangular  criss-cross mesh}
	\begin{tabular}{|c|c|c|c|}
	\hline
		$i$&$\lambda^{(i)}$&$\lambda_h^{(i)}$&$\big|\lambda^{(i)}-\lambda_h^{(i)}\big|$\\\hline
    1&   2&   2.000000009674193&9.674193002240372e-09\\
   2&   5&   5.000000169214921& 1.692149211862670e-07\\
   3&   5&   5.000000169214921& 1.692149211862670e-07\\
   4&   8&   8.000000618707245&6.187072454366671e-07\\
   5&  10&  10.000001465604944&1.465604944073107e-06\\
   6&  10&  10.000001465605909&1.465605908634871e-06\\
   7&  13&  13.000002783449016& 2.783449016163786e-06\\
   8&  13&  13.000002783451830& 2.783451829913020e-06\\
   9&  17&  17.000007469156596&7.469156596329185e-06\\
  10&  17&  17.000007469156596&7.469156596329185e-06\\
	\hline
	\end{tabular}
\label{table5}
\end{table}

\begin{table}[h]
	\caption{Example 1: convergence rate of the first eigenvalue $k = 3$ on rectangular criss-cross mesh}
	\begin{tabular}{|c|c|c|}
	\hline
		$h$&$\big|\lambda^{(1)}-\lambda_h^{(1)}\big|$&rate\\\hline
  $\pi/8$&  3.918771494682005e-05&-\\
  $\pi/16$& 2.468843251612896e-06  & 3.9885\\
  $\pi/32$& 1.546833066079500e-07  & 3.9964\\
  $\pi/64$& 9.674193002240372e-09&3.9990\\
  	\hline
	\end{tabular}
\label{table6}
\end{table}

\subsection{Example 2}
In this example, we compute eigenvalues of discrete problem \eqref{fem1} on an L-shape domain $\Omega =(0,\pi)\times(0,\pi)\backslash[\frac{\pi}{2},\pi)\times [\frac{\pi}{2},\pi)$ using quadratic Lagrange elements on a criss-cross mesh. Figure \ref{figure1} shows a rectangular criss-cross mesh of $\Omega$ with $h=\pi/8$.

In the third column of Table \ref{table7}, we show the first ten eigenvalues computed by quadratic Lagrange elements on the rectangular cirss-cross mesh for a fixed $h=\pi/64$.  In the second column, we list the first ten eigenvalues computed by the primal formulation \eqref{primal-v} with quadratic Lagrange elements on the same mesh. We can see from the table that the quadratic Lagrange elements on criss-cross meshes can produce correctly convergent eigenvalues for \eqref{vf} on the L-shape domain. 

\begin{table}[h]
	\caption{Example 2: first ten numerical eigenvalues with $k = 2$ and $h=\pi/80$ on rectangular mesh}
	\begin{tabular}{|c|c|c|}
	\hline
		$i$&$\lambda^{(i)}_{\operatorname{primal}}$&$\lambda_h^{(i)}$\\\hline
    1&     3.907542086020698 & 3.905354563577878\\
   2&  6.159216512440113&6.159213093492622\\
   3&   8.000000359186648& 8.000000253513941\\
   4&  11.964607939274400& 11.964606066013996\\
   5& 12.935434918397130&12.930090853126162\\
   6& 16.810290508710921&16.806272631957246\\
   7& 18.216953452465365&18.216933140031713\\
   8&   20.000007388434128&20.000004430751620\\
   9& 20.000007388434170& 20.000004430753396\\
  10& 22.984997819905693&  22.980510520283705 \\   
	\hline
	\end{tabular}
\label{table7}
\end{table}

    \section*{Acknowledgement}

The work of KH was supported by a Royal Society University Research Fellowship (URF$\backslash$ R1$\backslash$221398).

\bibliographystyle{plain}

\bibliography{lagrange-crisscross-v2}{}
~\\
\end{document}